\documentclass[a4paper]{article}
\usepackage{amssymb} 
\usepackage{amsmath} 
\usepackage{amsthm} 
\usepackage{mathtools}
\usepackage{hyperref}

\usepackage{verbatim}

\usepackage{xcolor}

\newcommand{\eps}{\varepsilon}

\newcommand{\avgdeg}{\overline{\Delta}}
\newcommand{\avgmult}{\overline{\mu}}

\renewcommand{\theenumi}{$(\roman{enumi})$}
\renewcommand{\labelenumi}{\theenumi}
\usepackage[shortlabels]{enumitem}
\setlist{nolistsep, noitemsep}

\newcommand{\keep}{\ensuremath\mathrm{Keep}}

\newcommand{\Prob}[1]{\ensuremath{%
    \mathbb P\left[#1\right]
  }}
\newcommand{\Expect}[1]{\ensuremath{%
    \mathbb E\left[#1\right]
  }}

\newcommand{\colouredVtcs}{\ensuremath A^{\mathrm{col}}}
\newcommand{\availableList}{\ensuremath L^\phi}
\newcommand{\remainingList}{\ensuremath L^\phi_{\mathrm{uncol}}}
\newcommand{\remainingColours}{\ensuremath H^\phi_{\mathrm{uncol}}}
\newcommand{\remainingColoursRV}{\ensuremath \mathbf{\#useable\mathunderscore cols}}
\newcommand{\removedColoursRV}{\ensuremath \mathbf{\#unuseable\mathunderscore cols}}
\newcommand{\oldColourDegRV}{\ensuremath \mathbf{remaining\mathunderscore cols'\mathunderscore old\mathunderscore deg}}
\newcommand{\relColourDegRV}{\ensuremath \mathbf{relevant\mathunderscore cols'\mathunderscore lost\mathunderscore deg}}

\newcommand{\totalActivatedNbrsRV}{\ensuremath \mathbf{\#activated\mathunderscore nbrs}}
\newcommand{\uncolActivatedNbrsRV}{\ensuremath \mathbf{\#uncoloured\mathunderscore nbrs}}
\newcommand{\colActivatedNbrsRV}{\ensuremath \mathbf{\#coloured\mathunderscore nbrs}}

\newcommand{\sampleSpace}{\ensuremath \mathbf\Omega}
\newcommand{\sigmaAlg}{\ensuremath \mathbf\Sigma}
\newcommand{\conflictsRV}{\ensuremath \mathbf{\#conflicts}}
\newcommand{\change}{\ensuremath \delta}

\newtheorem{theorem}{Theorem}
\newtheorem{proposition}[theorem]{Proposition}
\newtheorem{lemma}[theorem]{Lemma}

\newtheorem*{chernoff}{The Chernoff Bound}
\newtheorem*{slll}{The Lov\'asz Local Lemma}

\newtheorem{claim}[theorem]{Claim}

\title{Colourings, transversals and local sparsity}

\author{
Ross J. Kang
\thanks{Department of Mathematics, Radboud University Nijmegen, Netherlands. 
Email: \protect\href{mailto:ross.kang@gmail.com}{\protect\nolinkurl{ross.kang@gmail.com}}. Supported by a Vidi grant (639.032.614) of the Netherlands Organisation for Scientific Research (NWO).}
\and
Tom Kelly
\thanks{School of Mathematics, University of Birmingham, UK. 
Email: \protect\href{mailto:T.J.Kelly@bham.ac.uk}{\protect\nolinkurl{T.J.Kelly@bham.ac.uk}}. Partially supported by the EPSRC, grant no. EP/N019504/1}
}

\begin{document}

\maketitle

\begin{abstract}
Motivated both by recently introduced forms of list colouring and by earlier work on independent transversals subject to a local sparsity condition, we use the semi-random method to prove the following result.

For any function $\mu$ satisfying $\mu(d)=o(d)$ as $d\to\infty$, there is a function $\lambda$ satisfying $\lambda(d)=d+o(d)$ as $d\to\infty$ such that the following holds.
For any graph $H$ and any partition of its vertices into parts of size at least $\lambda$ such that
(a) for each part the average over its vertices of degree to other parts is at most $d$, and (b) the maximum degree from a vertex to some other part is at most $\mu$, there is guaranteed to be a transversal of the parts that forms an independent set of $H$.

This is a common strengthening of two results of Loh and Sudakov (2007) and Molloy and Thron (2012), each of which in turn implies an earlier result of Reed and Sudakov (2002).

\end{abstract}

\section{Introduction}\label{sec:intro}

Let $H$ be a graph with vertex partition $W_1, \dots, W_m$.  An \textit{independent transversal of $H$ with respect to $\{W_i\}_{i=1}^m$} is a collection $\{w_i\}_{i=1}^m$ of independent vertices in $H$ such that $w_i \in W_i$ for each $i \in \{1, \dots, m\}$.
Writing $\Delta(H)$ for the maximum degree of $H$, the following classic combinatorial question is essentially due to Erd\H{o}s (see~\cite{BES75}).
\begin{enumerate}
\renewcommand{\theenumi}{\Alph{enumi}}
\renewcommand{\labelenumi}{(\theenumi)}
\item\label{question} What is the least $\Lambda=\Lambda(d)$ such that, for every $H$ and $W_1, \dots, W_m$ as above satisfying moreover that $\Delta(H) \le d$ and $|W_i| \ge \Lambda$ for every $i \in \{1, \dots, m\}$, there is an independent transversal of $H$ with respect to $\{W_i\}_{i=1}^m$?
\end{enumerate}
Independently, Alon~\cite{Alo88} and Fellows~\cite{Fel90} showed that $\Lambda$ is linear in $d$, and later, in an acclaimed work, Haxell (see~\cite{Hax95,Hax01}) used topological methods to prove that $\Lambda(d) \le 2d$. In fact, $\Lambda(d)=2d$ for every $d$ as certified by an elementary construction due to Szab\'o and Tardos~\cite{SzTa06}.

Now let $G$ be a loopless multigraph, 
and let $L: V(G) \to 2^{V(H)}$ define a vertex partition of $V(H)$, i.e.~$\{L(v)\}_{v\in V(G)}$ defines a collection of disjoint subsets of $V(H)$ whose union comprises $V(H)$.
An {\em independent transversal of $H$ with respect to $L$} is
an independent transversal of $H$ with respect to $\{L(v)\}_{v \in V(G)}$.
We may assume without loss of generality that $H$ is a {\em cover graph for $G$ via $L$}: if $vv'\notin E(G)$ then the bipartite subgraph of $H$ induced between $L(v)$ and $L(v')$ is empty.
Viewed in this way, the independent transversals in $H$ may be related to vertex-colourings of $G$, as we now discuss.  

Any mapping $L: V(G)\to 2^{{\mathbb Z}^+}$ is called a {\em list-assignment} of $G$; a colouring $\phi$ of $V(G)$ is called an {\em $L$-colouring} if $\phi(v)\in L(v)$ for any $v\in V(G)$. 
The problem of finding proper $L$-colourings for various natural choices of $G$ is another famous combinatorial problem known as {\em list colouring}~\cite{ERT80,Viz76}. From $G$ and $L$ as above, we may produce a cover graph $H_\ell = H_\ell(G, L)$ for $G$ as follows. For every $v\in V(G)$, let $L_\ell(v)=\{(v, c)\}_{c\in L(v)}$. Let $V(H_\ell) = \cup_{v\in V(G)} L_\ell(v)$ and define $E(H_\ell)$ by letting $(v, c)(v', c') \in E(H_\ell)$ if and only if $vv'= e$ for some $e\in E(G)$ and $c=c'\in L(v)\cap L(v')$. Then independent transversals of $H_\ell$ with respect to $L_\ell$ are in one-to-one correspondence with proper $L$-colourings of $G$.

Question~\ref{question} with respect to $H_\ell$ was asked by Reed~\cite{Ree99}.  That is, what is the least $\Lambda_\ell = \Lambda_\ell(d)$ such that if a graph $G$ has a list-assignment $L$ satisfying $\Delta(H_\ell(G, L)) \leq d$ and $|L(v)| \geq \Lambda_\ell$ for every $v\in V(G)$, then $H_\ell$ has an independent transversal with respect to $L_\ell$?  Reed conjectured that $\Lambda_\ell(d) = d+1$. Reed and Sudakov~\cite{ReSu02} proved that $\Lambda_\ell(d) = d+o(d)$ as $d\to\infty$; however, Bohman and Holzman~\cite{BoHo02} disproved Reed's conjecture by exhibiting a construction certifying $\Lambda_\ell(d) \ge d+2$.

For $H$ being a cover graph for $G$ via $L$, we need the notion of \textit{maximum colour multiplicity $\mu_L(H)$ of $H$ with respect to $L$}, which is given by
\begin{equation*}
  \mu_L(H) \coloneqq \max_{vv' \in E(G), c \in L(v)} |N_H(c) \cap L(v')|.
\end{equation*}
If $G$ is a graph with list-assignment $L$, then $\mu_{L_\ell}(H_\ell(G, L)) \leq 1$.  Note that it makes no difference to $H_\ell$ whether $G$ is a multigraph or the underlying simple graph.
In 2005, Aharoni and Holzman (see~\cite{LoSu07}) asked Question~\ref{question} in the special case when $H$ is a cover graph for $G$ via $L$ satisfying $\mu_L(H) \leq 1$. 
In particular, what is the smallest $\Lambda_1=\Lambda_1(d)$  such that, if $H$ is a cover graph for $G$ via $L$ satisfying moreover that $\Delta(H) \le d$, $\mu_L(H)=1$, and $|L(v)| \geq \Lambda_1$ for every $v\in V(G)$, then $H$ has an independent transversal with respect to $L$?
Loh and Sudakov~\cite{LoSu07} resolved this problem asymptotically by showing that $\Lambda_1(d)=d+o(d)$ as $d\to\infty$. Furthermore they proved the same result under the milder assumption that $\mu_L(H) = o(d)$ as $d\to\infty$.  Since $\Lambda_\ell(d) \leq \Lambda_1(d)$ always, this also generalizes the aforementioned result of Reed and Sudakov.

This question can also be expressed in the framework of {\em correspondence colouring}~\cite{DvPo17} (also known as {\em DP-colouring}), a more general form of the list colouring problem that has recently captivated the graph colouring community.
A \textit{correspondence-assignment} for $G$ is a pair $(L, M)$ where $L$ is a list-assignment for $G$ and $M = \{M_e\}_{e \in E(G)}$ where $M_{e}$ is a matching between $\{v\}\times L(v)$ and $\{v'\}\times L(v')$ for each edge $e = vv'$.   An {\em $(L, M)$-colouring} of $G$ is an $L$-colouring $\phi$ of $G$ such that every edge $e = vv' \in E(G)$ satisfies $(v, \phi(v))(v', \phi(v')) \notin M_e$.  Note that an $(L, M)$-colouring is not necessarily a proper colouring of $G$.  Given a correspondence-assignment $(L, M)$, we may produce a cover graph $H_{DP} = H_{DP}(G, (L, M))$ for $G$ as follows.  For every $v\in V(G)$, let $L_{DP}(v) = \{(v, c)\}_{c\in L(v)}$, and let $V(H_{DP}) = \cup_{v\in V(G)} L_{DP}(v)$.  Define $E(H_{DP})$ by letting $(v, c)(v', c') \in E(H_{DP})$ if and only if $vv'=e$ for some $e \in E(G)$ such that $(v, c)(v', c') \in M_e$.  Then independent transversals of $H_{DP}$ with respect to $L_{DP}$ are in one-to-one correspondence with $(L, M)$-colourings of $G$.  Morover, if $G$ is a simple graph, then $\mu_{L_{DP}}(H_{DP}) \leq 1$, and whenever $H'$ is a cover graph for a simple graph $G'$ via $L'$ satisfying $\mu_{L'}(H') \leq 1$, there is a correspondence-assignment $(L, M)$ for $G'$ so that $H_{DP}(G', (L, M)) \cong H'$.  Thus, asking Question~\ref{question} with respect to $H_{DP}$ for simple $G$ is equivalent to asking what is $\Lambda_1$.

Note that in the special case for which the matchings $M_e$ ``recognise'' the colours, i.e.~$M_e = \{(v,i)(v',i)\}_{i\in L(v)\cap L(v')}$ if $e=vv'$ and $e\in E(G)$, then $H_{DP}$ is equivalent to $H_\ell$.
Note also that $\mu_{L_{DP}}(H_{DP})$ is at most $\mu(G)$, the maximum multiplicity of an edge in $G$.

\subsection*{Bounded average colour degrees}

In this work we consider Question~\ref{question} and the above narrative in a further strengthened form. For $H$ being a cover graph for $G$ via $L$, let us define the {\em maximum average colour degree $\avgdeg_L(H)$ of $H$ with respect to $L$} by
\[
\avgdeg_L(H) := \max_{v\in V(G)}\frac{1}{|L(v)|}\sum_{w\in L(v)} \deg_H (w).
\]
We remark that occasionally we will drop the subscripts in $\mu_L(H)$ and $\avgdeg_L(H)$ when the context is clear.
Motivated by a graph colouring problem, the following natural variation upon Question~\ref{question} was implicitly asked recently (in the alternative formulation of \textit{single-conflict chromatic number}, which we discuss later) by Dvo\v{r}\'ak, Esperet, Ozeki and the first author~\cite{DEKO18+}.

\begin{enumerate}
\renewcommand{\theenumi}{\Alph{enumi}}
\renewcommand{\labelenumi}{(\theenumi)}
\setcounter{enumi}{1}
\item\label{avgquestion} What is the least $\Lambda'=\Lambda'(d)$ such that, for every $H$ and $L$ as above satisfying moreover that $\avgdeg_L(H) \le d$ and $|L(v)| \ge \Lambda'$ for every $v\in V(G)$, there is an independent transversal of $H$ with respect to $L$?
\end{enumerate}
Note that since $\avgdeg_L(H) \le \Delta(H)$ always, we have $\Lambda'(d) \ge \Lambda(d) = 2d$. It was already observed that $\Lambda'(d) \le 4d$~\cite[Proposition~5]{DEKO18+}, and for convenience we restate this in Proposition~\ref{prop:upper4} below.

We can also ask Question~\ref{avgquestion} in the context of list colouring and correspondence colouring for simple $G$ as before, 
and our main result resolves both of these questions in a stronger form. 
More fully, our main result is an asymptotically optimal bound in Question~\ref{avgquestion} in the special case that $\mu_L(H)$ is a vanishingly small fraction of $\avgdeg_L(H)$.

\begin{theorem}\label{thm:main}
For any function $\mu$ satisfying $\mu(d)=o(d)$ as $d\to\infty$, there is a function $\lambda$ satisfying $\lambda(d) = d + o(d)$ as $d\to\infty$ such that the following holds.
For every $H$ being a cover graph for $G$ via $L$ satisfying
\begin{itemize}
\item $\avgdeg_L(H) \le d$,
\item $|L(v)| \ge \lambda$ for all $v\in V(G)$, and
\item $\mu_L(H) \le \mu$,
\end{itemize}
there is an independent transversal of $H$ with respect to $L$.
\end{theorem}
\noindent
Note that since $\avgdeg_L(H) \le \Delta(H)$ always, Theorem~\ref{thm:main} is stronger than the theorem of Loh and Sudakov, and is thus also stronger than the result of Reed and Sudakov. It also implies a more recent result of Molloy and Thron~\cite{MoTh12} on adaptable choosability (which itself also implies the result of Reed and Sudakov), which we now explain.

The ``single-conflict'' version of correspondence colouring a multigraph $G$ concerns correspondence-assignments $(L, M)$ for $G$ where the matchings $M_e$ for $e\in E(G)$ have size 1.  Equivalently, it concerns the existence of independent transversals in a graph $H$ where $H$ is a cover graph for $G$ via $L$ such that for every $vv'\in E(G)$, we have $|E_H(L(v), L(v'))| = \mu_G(vv')$, where $\mu_G(vv')$ is the multiplicity of the edge $vv'$ in $G$.  More precisely, the \textit{single-conflict chromatic number} of a multigraph $G$ is the smallest $k$ such that the following holds: for every correspondence-assignment $(L, M)$ satisfying $|L(v)| \geq k$ for every $v\in V(G)$ and $|M_e| = 1$ for every $e \in E(G)$, there is an $(L, M)$-colouring of $G$.  Importantly, in this case every $v\in V(G)$ satisfies $|L(v)|\cdot \avgdeg_L(H) = \Delta(G)$.  Thus, we could equivalently ask Question~\ref{avgquestion} for such graphs $H$ and replace $\avgdeg_L(H)$ with $\Delta(G)/\Lambda'$, and this is essentially the same as asking for the best bound on the single-conflict chromatic number of multigraphs of bounded maximum degree.  If we restrict the question further to the case where $G$ has a list-assignment $L'$ such that $H$ is isomorphic to a subgraph of $ H_\ell(G, L')$, then similarly we are asking for the best bound on the \textit{adaptable choosability} of multigraphs of bounded maximum degree.  The adaptable choosability of a multigraph $G$ is the smallest $k$ such that the following holds: for every correspondence-assignment $(L, M)$ satisfying $|L(v)| \geq k$ for every $v \in V(G)$ and $|M_e| = 1$ for every $e\in E(G)$ and moreover $(u, c)(v, c') \in M_{uv}$ only if $c = c'$, there is an $(L, M)$-colouring of $G$.  In this way, Molloy and Thron's bound on the adaptable choosability implies that if $G$ is a graph with list-assignment $L$ and $H \subseteq H_\ell(G, L)$ satisfying $\avgdeg_{L_\ell}(H) \leq d$ and $|L(v)| \geq d + o(d)$, 
then $H$ has an independent transversal with respect to $L_\ell$.  In this case, we still have $\mu_{L_\ell}(H) \leq 1$, so Theorem~\ref{thm:main} generalizes this result by allowing $H \subseteq H_{DP}(G, (L, M))$ for a correspondence-assignment $(L, M)$ satisfying $\mu_{L_{DP}}(H) = o(d)$.

As in~\cite{ReSu02}, the proof of Theorem~\ref{thm:main} proceeds through a semi-random procedure. We have additionally incorporated ideas from both~\cite{LoSu07} and~\cite{MoTh12} as well as modern concentration tools.

\subsection*{Structure of the paper}

In the next section, we present the probabilistic tools we require for the proof. We give an outline of the two-phase procedure in Section~\ref{sec:finish}. The bulk of the paper is devoted to the proof of the second, main phase of the procedure in Section~\ref{sec:semirandom}. At the end of the paper, we discuss a handful of interesting problems for further study.

\section{Probabilistic tools}\label{sec:tools}

We need several probabilistic tools.  The first such is the Lov\' asz Local Lemma.

\begin{slll}\label{local lemma}
Let $p\in[0,1)$ and $\mathcal A$ a finite set of events such that for every $A\in\mathcal A$,
\begin{enumerate}
	\item $\Prob{A} \leq p$, and
	\item $A$ is mutually independent of a set of all but at most $d$ other events in $\mathcal A$.
\end{enumerate}
If $4pd\leq 1$, then the probability that none of the events in $\mathcal A$ occur is strictly positive.
\end{slll}
\noindent
When we apply this, each bad event in $\mathcal A$ is an event in which a certain random variable deviates significantly from its expectation.

The remainder of this section is devoted to providing general sufficient conditions for a random variable to be concentrated around its expectation with high probability.  The first and most basic of these is the Chernoff Bound.

\begin{chernoff}\label{chernoff}
  If $X = \sum_{i=1}^nX_i$ is a sum of bounded independent random variables $a_i\leq X_i\leq b_i$, then
  \begin{equation*}
    \Prob{|X - \Expect{X}| \geq t} \leq 2\exp\left(-\frac{t^2}{\sum_{i=1}^n(b_i - a_i)^2}\right).
  \end{equation*}
  In particular, when $X_i$ are indicator variables (i.e. $a_i = 0$ and $b_i = 1$), we have
  \begin{equation*}
    \Prob{|X - \Expect{X}| \geq t} \leq 2\exp\left(-\frac{t^2}{n}\right).
  \end{equation*}
\end{chernoff}

The Chernoff Bound provides very tight concentration, but is limited in its applicability.  A much more flexible concentration inequality is Talagrand's Inequality~\cite{T95}.  It can be cumbersome though, so many researchers have proved derivations of it more suitable for combinatorial applications. 
We use the following version from~\cite{MR14}, see~\cite[Remark~1]{KePo20}.
\begin{theorem}[Molloy and Reed~\cite{MR14}]\label{MR-tala}
  Let $X$ be a non-negative random variable determined by the independent trials $T_1, \dots, T_n$.  Suppose that for every set of possible outcomes of the trials, we have that
  \begin{itemize}
  \item changing the outcome of any one trial can affect $X$ by at most $\change$; and
  \item for each $s > 0$, if $X \geq s$, then there is a set of at most $rs$ trials whose outcomes certify that $X \geq s$.
  \end{itemize}
  Then for any $t\geq 0$ where $t/2 \geq 20\change + \sqrt{r\Expect{X}} + 64\change^2r$, we have
  \begin{equation*}
    \Prob{|X - \Expect{X}| > t} \leq 4\exp\left(\frac{-t^2}{32\change^2 r(\Expect{X} + t)}\right).
  \end{equation*}
\end{theorem}

We also need a more robust version due to Bruhn and Joos~\cite[Theorem~7.5]{BrJo18}, which applies as long as almost all outcomes satisfy the conditions of Theorem~\ref{MR-tala}, that is, it takes into account a set of exceedingly unlikely exceptional outcomes.

We say a random variable has \textit{upward $(s, \change)$-certificates} with respect to a set of exceptional outcomes $\sampleSpace^*$ if for every $\omega\in\sampleSpace\setminus\sampleSpace^*$ and every $t > 0$, there exists an index set $I$ of size at most $s$ so that $X(\omega') \geq X(\omega) - t$ for any $\omega' \in \sampleSpace\setminus \sampleSpace^*$ for which the restrictions $\omega|_I$ and $\omega'|_I$ differ in at most $t/\change$ coordinates.
\begin{theorem}[Bruhn and Joos~\cite{BrJo18}]\label{BJ-tala}
  Let $((\sampleSpace_i, \sigmaAlg_i, \mathbb P_i))$ be probability spaces, let $(\sampleSpace, \sigmaAlg, \mathbb P)$ be their product space, and let $\sampleSpace^*\subseteq\sampleSpace$ be a set of exceptional outcomes.  Let $X : \sampleSpace \rightarrow \mathbb R$ be a non-negative random variable, and let $M = \max\{\sup X, 1\}$, and let $\change \geq 1$.  If $\Prob{\sampleSpace^*} \leq M^{-2}$ and $X$ has upward $(s, \change)$-certificates, then for $t > 50 \change\sqrt s$,
  \begin{equation*}
    \Prob{|X - \Expect{X}| \geq t} \leq 4\exp\left(-\frac{t^2}{16\change^2s}\right) + 4\Prob{\sampleSpace^*}.
  \end{equation*}
\end{theorem}

\section{A two-phase semi-random procedure}\label{sec:finish}

Dvo{\v{r}}{\'a}k, Esperet, Ozeki and the first author observed~\cite[Proposition~5]{DEKO18+} using the Lov\' asz Local Lemma that every multigraph of maximum degree $\Delta$ has {\em single-conflict chromatic number} at most $\lceil \sqrt{e(2\Delta - 1)}\rceil$. This implies that
for every $H$ being a cover graph for $G$ via $L$ satisfying
 $\avgdeg_L(H) \leq d$ and $|L(v)| \geq 2ed$, there is an independent transversal of $H$ with respect to $L$.  They remarked that by using the Local Cut Lemma~\cite{Ber17} instead of the Lov\' asz Local Lemma, one can improve the bound $\lceil\sqrt{e(2\Delta - 1)}\rceil$ to $2\sqrt\Delta$. This translates as follows.

\begin{proposition}\label{prop:upper4}
If $H$ is a cover graph for $G$ via $L$ satisfying
\begin{itemize}
\item $\avgdeg_L(H) \le d$ and
\item $|L(v)| \ge 4d$ for all $v\in V(G)$,
\end{itemize}
then
there is an independent transversal of $H$ with respect to $L$.
\end{proposition}
\noindent
Proposition~\ref{prop:upper4} suffices as the ``finishing blow'' in our proof of Theorem~\ref{thm:main}.  We reduce Theorem~\ref{thm:main} to Proposition~\ref{prop:upper4} using a two-phase semi-random procedure.
We note that the value of the constant $4$ is unimportant: our reduction also holds if $4$ is replaced by $4e$.

The first phase reduces the problem from one in which $\mu_L(H) = o(d)$ to one in which $\mu_L(H) \leq d^{1/5}$, and this phase is embodied by the following result.

\begin{theorem}\label{list-reduction-lemma}
  For every $d_1, \eps > 0$, there exists $\gamma_0, d_0 > 0$ such that following holds for all $\gamma < \gamma_0$ and $d > d_0$.  
For every $H$ being a cover graph for $G$ via $L$ satisfying
  \begin{itemize}
  \item $|L(v)| \geq (1 + \eps)d$ for all $v\in V(G)$,
  \item $\avgdeg_L(H) \leq d$, and
  \item $\mu_L(H) \leq \gamma d$,
  \end{itemize}
  there exists $d' \geq d_1$ and an induced subgraph $H' \subseteq H$
that is a cover graph for $G$ via $L'$ for some $L'$ satisfying
  \begin{itemize}
  \item $|L'(v)|  \geq (1 + \eps/2)d'$ for all $v\in V(G)$,
  \item $\avgdeg_{L'}(H') \leq d'$,
  \item $\mu_{L'}(H') \leq d'^{1/5}$, and
  \item $\Delta(H') \leq d'\log^{1/2} d'$.
  \end{itemize}
\end{theorem}

Without the requirement that $\Delta(H') \leq d'\log^{1/2}d'$, the proof of Theorem~\ref{list-reduction-lemma} can be obtained from the proof of~\cite[Theorem~3.1]{LoSu07} with the following substitutions, letting $V(G) = \{v_1, \dots, v_r\}$:
\begin{itemize}
\item $\Delta \rightarrow d$,
\item $G\rightarrow H$,
\item $V_1, \dots, V_r \rightarrow L(v_1), \dots, L(v_r)$, and
\item ``local degree'' $\rightarrow \mu_L(H)$.
\end{itemize}
Effectively, the main difference is that we use the maximum average colour degree $\avgdeg_L(H)$ instead of $\Delta_L(H)$, and we obtain the weaker conclusion $\mu_{L'}(H') \leq d'^{1/5}$ rather than $\mu_{L'}(H') \leq 10$.  The reason for this difference is that~\cite[Lemma~3.2]{LoSu07} does not hold with these substitutions.  However,~\cite[Lemma~3.3]{LoSu07} does, so we follow the proof of~\cite[Theorem~3.1]{LoSu07}, iteratively applying~\cite[Lemma~3.3]{LoSu07} until the point at which it is possible to apply~\cite[Lemma~3.2]{LoSu07}.  To obtain the requirement $\Delta(H') \leq d'\log^{1/2} d'$, we simply observe that for each $v \in V(G)$, the number of colours $c\in L'(v)$ with $\deg_H(c) > d'\log^{1/2} d'$ is a vanishingly small fraction of $|L'(v)|$, so we delete them, as in~\cite[Proposition~4.1]{MoTh12}.

Since the proof of Theorem~\ref{list-reduction-lemma} so closely resembles the proofs of these other results, we defer it to the appendix.  For most of what remains of the paper, we focus on the second phase of our semi-random procedure.

For convenience, we introduce some further notation. If 
$H$ is a cover graph for $G$ via $L$
then we say that an {\em $(L, H)$-colouring} is a colouring $\phi$ of $V(G)$ such that $\phi(v) \in L(v)$ for every vertex $v\in V(G)$, and $\phi$ is {\em proper} if $\{\phi(v) : v \in V(G)\}$ is an independent transversal of $H$ with respect to $L$.  If $C\subseteq V(G)$ and $\phi$ is a colouring with domain $C$ such that $\phi(v) \in L(v)$ for each $v\in C$, then we say $\phi$ is a \textit{partial $(L, H)$-colouring}, and 
\begin{itemize}
\item we say $v\in C$ is \textit{$\phi$-coloured} and \textit{$\phi$-uncoloured} otherwise,
\item we say $c\in V(H)$ is \textit{$\phi$-unuseable} if $c\phi(v) \in E(H)$ for some $v \in C$ and $c$ is \textit{$\phi$-useable} otherwise,
\item we let $H^\phi = H\setminus \cup_{v\in S} N_H(\phi(v))$, that is, the graph induced by $H$ on the $\phi$-useable colours,
\item we let $L^\phi(v) = L(v) \cap V(H^\phi)$ for each $v \in V(G)$, that is, the set of $\phi$-useable colours in $L(v)$,
\item we let $\remainingColours = H^\phi - \cup_{v\in C}L(v)$, that is, the graph induced by $H$ on the $\phi$-useable colours in the list of a $\phi$-uncoloured vertex,
\item and we let $\remainingList = L^\phi|_{V(G)\setminus C}$.
\end{itemize}
If additionally $\{\phi(v) : v\in C\}$ is an independent set in $H$, then we say $\phi$ is proper.  
It is important to notice that, if $\phi$ is a proper partial $(L, H)$-colouring and $G - C$ has an $(\remainingList,\remainingColours)$-colouring, then $H$ has an independent transversal with respect to $L$.

In the second, main phase, we find a sequence of proper partial $(L, H)$-colourings $\phi$ of $G$ in which we gradually improve the ratio of $|L(v)|/\avgdeg_L(H)$ from $(1 + \eps)$ for each $v\in V(G)$ to 4 for each $\phi$-uncoloured vertex, at which point we can apply Proposition~\ref{prop:upper4}. The second phase is embodied by the following theorem.

\begin{theorem}\label{semirandom-partial-colouring-thm}
For every $\eps>0$ and every $H$ being a cover graph for $G$ via $L$ satisfying
  \begin{itemize}
  \item $|L(v)| \geq (1 + \eps)d$ for all $v\in V(G)$,
  \item $\avgdeg_L(H) \leq d$,
  \item $\mu_L(H) \leq d^{1/5}$, and
  \item $\Delta(H) \leq d\log^{1/2}d$
  \end{itemize}
  for $d$ sufficiently large,
  there is a proper partial $(L, H)$-colouring $\phi : C\subseteq V(G) \rightarrow V(H)$ of $H$ and an induced subgraph $H' \subseteq \remainingColours$ 
  that is a cover graph for $G-C$ via $L'$ for some $L'$
  such that $|L'(v)| \ge 4 \avgdeg_{L'}(H')$ for every $v \in V(G)\setminus C$.
\end{theorem}

We prove Theorem~\ref{semirandom-partial-colouring-thm} in Section~\ref{sec:semirandom}.  Our proof of Theorem~\ref{semirandom-partial-colouring-thm} incorporates ideas from both~\cite{LoSu07} and~\cite{MoTh12}.  We find the partial colouring $\phi$ in several iterations. Each iteration slightly improves the ratio of $|L(v)|/\avgdeg(H)$ without affecting the other parameters too much, so that we can proceed for $\Theta(\log d)$ iterations.  The main hurdle is that $\mu_L(H)$ can be relatively large, which affects the concentration of our random variables, but we can overcome this difficulty using Theorem~\ref{BJ-tala}, the ``exceptional outcomes'' version of Talagrand's Inequality.

We conclude this section with a proof of Theorem~\ref{thm:main} assuming Theorem~\ref{semirandom-partial-colouring-thm}.

\begin{proof}[Proof of Theorem~\ref{thm:main}]
  By Theorem~\ref{list-reduction-lemma}, it suffices to prove for any $\eps>0$ that for sufficiently large $d'$, if we have 
  a cover graph $H'$ for $G$ via $L'$
  satisfying $|L'(v)| \geq (1 + \eps)d'$, $\avgdeg_{L'}(H') \leq d'$, $\mu(H') \leq d'^{1/5}$, and $\Delta(H') \leq d'\log^{1/2}d'$, then there is an independent transversal of $H'$ with respect to $L'$. This is because since $H'$ is an induced subgraph of $H$, an independent transversal of $H'$ with respect to $L'$ must also be an independent transversal of $H$ with respect to $L$.

  Now by Theorem~\ref{semirandom-partial-colouring-thm}, there is a proper partial $(L', H')$-colouring $\phi : C\subseteq V(G)\rightarrow V(H')$ of $H$ and an induced subgraph $H'' \subseteq (H')^\phi_{\mathrm{uncol}}$ 
  that is a cover graph for $G-C$ via $L''$
  such that $|L''(v)| \ge 4\avgdeg_{L''}(H'')$ for every $v \in V(G)\setminus C$.  By Proposition~\ref{prop:upper4}, there is an independent transversal of $H''$ with respect to $L''$, or equivalently, $G - C$ has an $(L'', H'')$-colouring.  By combining an $(L'', H'')$-colouring of $G - C$ with $\phi$, we obtain an $(L', H')$-colouring of $G$, so $H'$ has an independent transversal with respect to $L'$, as required.
\end{proof}

\section{The main phase}\label{sec:semirandom}

We prove Theorem~\ref{semirandom-partial-colouring-thm} by applying several iterations of the following lemma.

\begin{lemma}\label{nibble-lemma}
  For every $\eps > 0$, the following holds for sufficiently large $d$.
If $H$ is a cover graph for $G$ via $L$ satisfying
  \begin{itemize}
  \item $\avgdeg_L(H) \leq d$,
  \item $\Delta(H) \leq d\log d$,
  \item $|L(v)| = \lceil\Lambda\rceil$ for all $v\in V(G)$, where $(1 + \eps)d \le \Lambda \le 4d$, and
  \item $\mu_L(H) \leq d^{1/4}$, 
  \end{itemize}
  and $\log^{-1} d \geq p \geq \log^{-2}d$,
  then there is a proper partial $(L, H)$-colouring $\phi : C\subseteq V(G) \rightarrow V(H)$ of $H$ and an induced subgraph $H' \subseteq \remainingColours$ 
  that is a cover graph for $G-C$ via $L'$ for some $L'$
  such that 
  \begin{itemize}
  \item $\avgdeg_{L'}(H') \leq \left(1 - \frac{p}{1 + \eps/4}\right)d$ and
  \item $|L'(v)| = \left\lceil\left(1 - \frac{p}{1 + 3\eps/4}\right)\Lambda\right\rceil$ for every $v\in V(G)\setminus C$.
  \end{itemize}
\end{lemma}

Before proving Lemma~\ref{nibble-lemma}, we first show how iteration of this lemma yields Theorem~\ref{semirandom-partial-colouring-thm}.  As demonstrated in the proof below, applying this lemma improves the ratio of list size to maximum average colour degree, since crucially $\left.(1 - \frac{p}{1 + 3\eps/4})\big/(1 - \frac{p}{1 + \eps/4})\right. \geq 1 + \eps p / 4$. 

\begin{proof}[Proof of Theorem~\ref{semirandom-partial-colouring-thm}]
  Let $p = \log^{-1} d$, $d_0 = d$ and $\Lambda_0 = (1 + \eps)d$, and for each integer $0\le i \le 12/(\eps p)$, let
  \begin{align*}
    d_{i + 1} = \left(1 - \frac{p}{1 + \eps/4}\right)d_i && \mathrm{and} && \Lambda_{i + 1} = \left(1 - \frac{p}{1 + 3\eps/4}\right)\Lambda_i.
  \end{align*}
  Note that for every integer $0\le i \le 12/(\eps p)$, we have
  \begin{equation*}
    d_{i} \ge \left(1 - \frac{p}{1 + \eps/4}\right)^{12/(\eps p)}d 
    = \Omega(d),
  \end{equation*}
  so we may always assume $d_i$ is large enough to apply Lemma~\ref{nibble-lemma}.  We may moreover assume $d\log^{1/2} d \leq d_{i}\log d_i$, $d^{1/5} \leq d_i^{1/4}$, and $p \geq \log^{-2}d_i$.
  Since
   \begin{equation*}
    \frac{1 - \frac{p}{1 + 3\eps/4}}{1 - \frac{p}{1 + \eps/4}} 
    = 1 + \frac{\eps/2}{(1 + 3\eps/4)(1 + \eps/4)p^{-1} - (1 + 3\eps/4)} \geq 1 + \frac{\eps p}{4},
  \end{equation*}
  for every integer $1\le i \le 12/(\eps p)$ we have
  \begin{equation*}
    \frac{\Lambda_{i}}{d_{i}} \geq \left(1 + \frac{\eps p}{4}\right)\frac{\Lambda_{i-1}}{d_{i-1}}
  \end{equation*}
  and moreover
  \begin{equation*}
    \frac{\lceil \Lambda_{\lfloor 12/(\eps p)\rfloor} \rceil}{d_{\lfloor 12/(\eps p)\rfloor}}
    \geq \left(1 + \frac{\eps p}{4}\right)^{\lfloor 12/(\eps p) \rfloor}\frac{\Lambda_0}{d_0}
    = \left(1 + \frac{\eps p}{4}\right)^{\lfloor 12/(\eps p) \rfloor}(1+\eps) \geq 4.
  \end{equation*}
  In particular, there exists an integer $1 \leq i^* \leq 12 / (\eps p)$ such that $\lceil\Lambda_{i^*}\rceil / d_{i^*} \geq 4$ and
  $\lceil\Lambda_{i^* - 1}\rceil / d_{i^* - 1} \leq 4$.  (We may assume $\lceil\Lambda_{0}\rceil / d_{0} \leq 4$ as otherwise $C = \varnothing$, $H' = H$, and $L' = L$ satisfy the theorem).
  
  We assume without loss of generality that $|L(v)| = \lceil\Lambda_0\rceil$ for every $v \in V(G)$, since we can truncate each list until equality holds by removing from $H$ those colours $c \in L(v)$ for which $|N_H(c) \cap (V(H)\setminus L(v))|$ is largest without increasing the maximum average colour degree.  Now let $H_0 = H$, $G_0 = G$ and $L_0 = L$. Due to the above calculations, inductively by Lemma~\ref{nibble-lemma}, for each integer $0\le i < i^*$, there is a proper partial $(L_i, H_i)$-colouring $\phi_{i + 1} : C_{i + 1} \subseteq V(G_i) \rightarrow H_{i}$ and an induced subgraph $H_{i + 1}$
  that is a cover graph for $G_{i+1}=G_i-C_{i+1}$ via $L_{i+1}$ for some $L_{i+1}$
  satisfying
  \begin{itemize}
  \item $\avgdeg_{L_{i + 1}}(H_{i + 1}) \leq d_{i + 1}$ and
  \item $|L_{i + 1}(v)| = \lceil \Lambda_{i + 1}\rceil \ge (1+\eps) d_{i+1}$ (by above).
  \end{itemize}
  
  Let $C = \cup_{i=1}^{i^*} C_i$, let $\phi(v) = \phi_i(v)$ if $v \in C_i$ for some integer $1\le i \le i^*$, let $H' = H_{i^*}$, and let $L' = L_{i^*}$.  By construction, $\phi :C \rightarrow V(H)$ is a proper partial $(L, H)$-colouring and $H'\subseteq \remainingColours$ is a cover graph for $G-C$ via $L'$.  Finally, by the choice of $i^*$, we have (for $d$ large enough)
  \begin{equation*}
    \frac{|L'(v)|}{\avgdeg_{L'}(H')} 
    \ge \frac{\lceil \Lambda_{i^*} \rceil}{d_{i^*}}
     \geq 4. \qedhere
  \end{equation*}
\end{proof}

It only remains to prove Lemma~\ref{nibble-lemma}.  For the remainder of this section, let $\eps > 0$ and $d$ sufficiently large, let $p$ satisfy $\log^{-1} d \geq p \geq \log^{-2}d$, and let $H$ be a cover graph for $G$ via $L$ satisfying the hypotheses of Lemma~\ref{nibble-lemma}.
Throughout the proof we assume that 
if $vv' \in E(G)$, then there exists $c \in L(v)$ and $c'\in L(v')$ such that $cc'\in E(H)$.  Thus $\Delta(G) \leq \Lambda d\log d$.
Even when we do not explicitly state it, we will always assume that $d$ is sufficiently large for certain inequalities to hold.

We will analyse a random proper partial $(L, H)$-colouring and use the Lov\' asz Local Lemma to show that with nonzero probability it satisfies the properties we desire.  Let us now describe this random colouring.

A \textit{wasteful $(L, H)$-colouring} is a pair $(A, \phi)$ where
\begin{itemize}
\item $A \subseteq V(G)$ is a set of \textit{activated vertices} and
\item $\phi$ is a partial $(L, H)$-colouring of $G$ with domain $A$.
\end{itemize}
Note that if $(A, \phi)$ is a wasteful $(L, H)$-colouring, then $\phi$ is not necessarily proper.
We let $\colouredVtcs$ be the set of vertices $v\in A$ with no neighbor $u\in A$ such that $\phi(v)\phi(u) \in E(H)$. 

 To prove Lemma~\ref{nibble-lemma}, we find a wasteful colouring $(A, \phi)$ such that every $v\in V(G)$ satisfies
\begin{itemize}
\item $\deg_{H^\phi}(c) \leq (1 - p + o(p))d$ for every $c\in L^\phi(v)$,
\item $|L^\phi(v)| \geq (1 - \frac{p + o(p)}{1 + \eps})\Lambda$, and
\item $\sum_{c\in L^\phi(v)}\deg_{H}(c) \leq (1 - \frac{p + o(p)}{1 + \eps})d\Lambda$.
\end{itemize}
In this case, we show that $\phi|_{\colouredVtcs}$ satisfies the conclusion of Lemma~\ref{nibble-lemma} where $C = \colouredVtcs$, $H' = H^\phi - \cup_{v\in\colouredVtcs} L(v)$, and $L' = L^\phi|_{V(G)\setminus \colouredVtcs}$. 
 We call the colouring wasteful because there may be $\phi|_{\colouredVtcs}$-useable colours not in $H'$ that we do not use.

The \textit{wasteful random colouring procedure with activation probability $p$} samples a wasteful $(L, H)$-colouring $(A, \phi)$ as follows:
\begin{itemize}
\item For each $v\in V(G)$, activate $v$ (i.e. let $v\in A$) with probability $p$.
\item For each $v\in A$, choose $\phi(v) \in L(v)$ uniformly at random.
\end{itemize}

In the analysis of this procedure, it will be helpful to
 define the following random variables for each vertex $v\in V(G)$ and $c\in L(v)$:
\begin{itemize}
\item $\remainingColoursRV_v(A, \phi) = |\availableList(v)|$, the number of $\phi$-useable colours in $L(v)$,
\item $\removedColoursRV_v(A, \phi) = |L(v)\setminus\availableList(v)|$, the number of $\phi$-unuseable colours in $L(v)$,
\item $\colActivatedNbrsRV_{v, c}(A, \phi) = \left|N_H(c) \cap \left(\cup_{u \in \colouredVtcs}L(u)\right)\right|$, 
\item $\totalActivatedNbrsRV_{v, c}(A, \phi) = \left|N_H(c) \cap \left(\cup_{u \in A}L(u)\right)\right|$, and
\item $\uncolActivatedNbrsRV_{v, c}(A, \phi) = \left|N_H(c) \cap \left(\cup_{u \in A\setminus\colouredVtcs}L(u)\right)\right|$.
\end{itemize}

Most of the proof is devoted to bounding the expected values of these random variables and showing that they are concentrated around their expectation with high probability.  For proving concentration, it will be convenient to assume that $\phi = \psi|_A$ where $\psi$ is a (not necessarily proper) $(L, H)$-colouring of $G$ where $\psi(v) \in L(v)$ is chosen uniformly at random for each $v\in V(G)$.  In this way, the random variables above are determined by $2|V(G)|$ independent random trials, half of which determine which vertices in $G$ are in $A$ and half of which determine which colour is assigned to each vertex of $G$ by $\psi$, which allows us to apply Theorems~\ref{MR-tala} and~\ref{BJ-tala}.

Our first step is to bound the expected values of some of these random variables.
To this end, let
\begin{equation*}
  \keep(v, c) = \Prob{c \in L^\phi(v)} = (1 - p/\Lambda)^{\deg_H(c)}.
\end{equation*}
By convexity of the exponential function and Jensen's Inequality, for $v\in V(G)$,
\begin{align}
  \label{convexity-of-keep}
  \sum_{c \in L(v)}\keep(v, c) &= \sum_{c \in L(v)}\left(1 - \frac{p}{\Lambda}\right)^{\deg_H(c)} \geq \left(1 - \frac{p}{\Lambda}\right)^{\sum_{c\in L(v)}\deg_H(c)/\Lambda}\Lambda \nonumber\\
  & \ge \left(1 - \frac{p}{\Lambda}\right)^d\Lambda
   \geq \left(1 - \frac{p}{1 + \eps}\right)\Lambda.
\end{align}
\begin{claim}\label{expectations-lemma}
  Every vertex $v\in V(G)$ satisfies
  \begin{equation}\label{expected-list-lize}
    \Expect{\remainingColoursRV_v} \geq \left(1 - \frac{p}{1 + \eps}\right) |L(v)|,
  \end{equation}
  and for every $c\in L(v)$, we have
  \begin{equation}
    \label{expected-number-of-coloured-nbrs}
    \Expect{\colActivatedNbrsRV_{v, c}} \geq p\left(1 - \frac{p}{1 + \eps}\right)\deg_H(c).
  \end{equation}
\end{claim}
\begin{proof}
  First~\eqref{expected-list-lize} follows from~\eqref{convexity-of-keep} and Linearity of Expectation.

  Now we prove~\eqref{expected-number-of-coloured-nbrs}.
  By Linearity of Expectation, we have
  \begin{equation*}
    \Expect{\colActivatedNbrsRV_{v, c}} = \sum_{(u, c')\in N_H(c)}\frac{p}{|L(u)|}\sum_{c''\in L(u)}\keep(u, c''),
  \end{equation*}
  and~\eqref{expected-number-of-coloured-nbrs} follows from the above equality combined with~\eqref{convexity-of-keep}, since $\Lambda = |L(u)|$ for every $u\in V(G)$.
\end{proof}

Now we need to show that the random variables in Claim~\ref{expectations-lemma} are close to their expectation with high probability.  We use Theorem~\ref{BJ-tala}, the exceptional outcomes version of Talagrand's Inequality.  To that end, we define an exceptional outcome for each vertex and show that it is unlikely.
First, for each vertex $u\in V(G)$ and $c\in L(u)$, we define
\begin{itemize}
\item $\conflictsRV_{u, c}(A, \phi) = \left|\left\{w \in N_G(u)\cap A : \phi(w) \in N_H(c)\right\}\right|$.
\end{itemize}
Then, for each vertex $v\in V(G)$, we define
\begin{itemize}
\item $\sampleSpace^*_{v} = \{(A, \phi) : \exists u\in N^2_G(v),\exists c\in L(u),\ \conflictsRV_{u, c}(A, \phi) \geq \log^2 d \}$.
\end{itemize}
The events $\sampleSpace^*_v$ include more outcomes than is necessary, but it is simpler to define it as we have.  Now we bound the probability of these exceptional events.

\begin{claim}\label{bad-event-prob}
  Every vertex $v\in V(G)$ satisfies
  \begin{equation*}
    \Prob{\sampleSpace^*_{v}} \leq 16e^3d^5\left(\frac{e}{\log d}\right)^{\log^2 d - 3}.
  \end{equation*}
\end{claim}
\begin{proof}
  First we let $u \in V(G)$ and $c\in L(u)$ and bound the probability that $\conflictsRV_{u, c}$ is too large:
  \begin{equation*}
    \Prob{\conflictsRV_{u, c} \geq \log^2 d} \leq \sum_{i = \lceil \log^2 d\rceil}^{\deg_H(c)}\binom{\deg_H(c)}{i}\left(\frac{1}{\Lambda}\right)^i.
  \end{equation*}
  By applying the bound $\binom{\deg_H(c)}{i} < \left(\frac{e\deg_H(c)}{i}\right)^i$ and using the fact that $\deg_H(c) \leq d\log d \leq \Lambda \log d$, the righthand side of the above inequality is at most
  \begin{equation*}
    \sum_{i=\lceil \log^2 d\rceil}^{\deg_H(c)}\left(\frac{e\deg_H(c)}{i\Lambda }\right)^i \leq \sum_{i=\lceil \log^2 d\rceil}^{\deg_H(c)}\left(\frac{e\log d}{i}\right)^i.
  \end{equation*}
  Since each term in the sum is at most $(e/\log d)^{\log^2 d}$ and there are at most $d\log d$ terms, it follows that
  \begin{equation*}
    \Prob{\conflictsRV_{u, c} \geq \log^2 d} \leq ed\left(\frac{e}{\log d}\right)^{\log^2 d - 1}.
  \end{equation*}
  Since $\Delta(G) \leq \Lambda d\log d$, there are at most $16d^4\log^2 d$ vertices $u\in N^2_G(v)$.  Thus, combining the above inequality with the Union Bound, we have
  \begin{equation*}
    \Prob{\sampleSpace^*_v} \leq \left(16d^4\log^2 d\right)\cdot ed\left(\frac{e}{\log d}\right)^{\log^2 d - 1} \leq 16e^3d^5\left(\frac{e}{\log d}\right)^{\log^2 d - 3},
  \end{equation*}
  as required.
\end{proof}

Having bounded the probability of the exceptional outcomes, we can now prove concentration of the random variables in Claim~\ref{expectations-lemma}.

\begin{claim}\label{concentration-lemma}
  For $d$ sufficiently large, every vertex $v\in V(G)$ satisfies
  \begin{equation}
    \label{list-size-concentration}
    \Prob{\left|\remainingColoursRV_v - \Expect{\remainingColoursRV_v}\right| > d^{5/6}} \leq \exp\left(-d^{1/7}\right),
  \end{equation}
  and for every $c\in L(v)$, we have
  \begin{equation}
    \label{colour-degree-concentration}
    \Prob{\left|\colActivatedNbrsRV_{v,c} - \Expect{\colActivatedNbrsRV_{v,c}}\right| > d^{5/6}} \leq \exp\left(-d^{1/7}\right).
  \end{equation}
\end{claim}
\begin{proof}
  First we prove \eqref{list-size-concentration}.  Since
  \begin{equation*}
    \remainingColoursRV_{v}(A, \phi) = |L(v)| - \removedColoursRV_{v}(A, \phi),
  \end{equation*}
  it suffices to show that $\removedColoursRV_{v, c}$ is concentrated.  To that end, we show that $\removedColoursRV_{v, c}$ has upward $(s, \change)$-certificates with respect to $\sampleSpace^* = \varnothing$ where $s = 2d\log d$ and $\change = d^{1/4}$.

  Let $(A, \phi)$ be a wasteful colouring.  We construct a bipartite subgraph $F \subseteq H$ with bipartition $(F_1, F_2)$ called the \textit{certificate graph}, as follows.  First, let $F_1 = L(v)\setminus L^\phi(v)$.  Now, for each $c\in F_1$, there is a pair $u, c'$ such that $c'\in N_H(c)$, $u\in A$ and $\phi(u) = c'$ certifying that $c\notin L^\phi(v)$.  For each such $c$, we choose one such pair $u, c'$ arbitrarily to put in $F_2$, and we add the edge between $c \in F_1$ and $(u, c')\in F_2$ to the certificate graph $F$.  We let $I$ index the trials determining that $u\in A$ and $\phi(u) = c'$ for each $(u, c')\in F_2$.  Note that $|I| \leq 2|L(v)| \leq s$, as required.

  If $(A', \phi')$ differs in at most $t/\change$ trials from $(A, \phi)$, then there is a set $F'_2 \subseteq F_2$ of size at least $|F_2| - t/\change$ such that every $(u, c') \in F'_2$ satisfies $u\in A'$ and $\phi'(u) = \phi(u)$.  Let $C'_2 = \{c' : (u, c') \in F'_2\}$, and let $F'\subseteq F$ be the induced subgraph of $F$ with bipartition $(L(v)\cap N_H(C'_2), F'_2)$.  Since $\mu(G) \leq d^{1/4}$, each $(u, c') \in F'_2$ has degree at most $d^{1/4}$ in $F$, and thus $|L(v) \cap N(C'_2)| \geq |F_1| - t$.  Hence,
  \begin{equation*}
    \removedColoursRV_{v}(A', \phi') \geq \removedColoursRV_{v}(A, \phi) - t,
  \end{equation*}
  and therefore $\removedColoursRV_{v}$ has upward $(s, \change)$-certificates, as claimed.

  Now by Theorem~\ref{BJ-tala} applied with $t = d^{5/6}$, we have
  \begin{multline*}
    \Prob{\left|\removedColoursRV_{v} - \Expect{\removedColoursRV_{v}}\right| \geq d^{5/6}} \\
    \leq 4\exp\left(-\frac{d^{10/6}}{32d^{3/2}\log d}\right) \leq \exp\left(-d^{1/7}\right),
  \end{multline*}
  and~\eqref{list-size-concentration} follows.
  
  Now we prove~\eqref{colour-degree-concentration}.  Since
  \begin{multline*}
    \colActivatedNbrsRV_{v, c}(A, \phi) = \totalActivatedNbrsRV_{v, c}(A, \phi) \\ - \uncolActivatedNbrsRV_{v, c}(A, \phi),
  \end{multline*}
  it suffices to show that both $\totalActivatedNbrsRV_{v, c}$ and $\uncolActivatedNbrsRV_{v, c}$ are concentrated.

  Since $\totalActivatedNbrsRV_{v, c}$ is simply the sum of $\deg_H(c)$ indicator variables, by the Chernoff Bound, we have
  \begin{multline}\label{active-degree-concentration}
    \Prob{\left|\totalActivatedNbrsRV_{v, c} - \Expect{\totalActivatedNbrsRV_{v, c}}\right| \geq \frac{1}{2}d^{5/6}}
    \\ \leq 2\exp\left(-\frac{d^{10/6}}{4\deg_H(c)}\right) \leq 2\exp\left(-\frac{d^{2/3}}{4\log d}\right) \leq \frac{1}{2}\exp\left(-d^{1/7}\right).
  \end{multline}

  Now we claim that $\uncolActivatedNbrsRV_{v, c}$ has upward $(s, \change)$-certificates with respect to $\sampleSpace^*_{v, c}$ where $s = 4d\log d$ and $\change = d^{1/4}\log^2 d$.  To that end, let $(A, \phi)\notin\sampleSpace^*_{v,c}$ be a wasteful colouring.  We construct an auxiliary bipartite graph $F$ with bipartition $(F_1, F_2)$ called the certificate graph, as follows.  First, let $F_1$ be the set of colours $c' \in N_H(c)$ where $c' \in L(u)$ for some $u \in A \setminus \colouredVtcs$.  For each $c' \in F_1$, there is a vertex $w\in V(G)$ with colour $\phi(w) \in N_H(\phi(u))$ certifying that $u\in A\setminus \colouredVtcs$.  For each such $c'$, we choose one such colour $\phi(w)$ arbitrarily and put $(\phi(u), \phi(w)) \in F_2$, and we add the edge between $c' \in F_1$ and $(\phi(u), \phi(w)) \in F_2$ to the certificate graph $F$.  We let $I$ index the trials determining that $u, w\in A$ and determining $\phi(u)$ and $\phi(w)$.  Note that $|I| \leq 4\deg_H(c) \leq 4d\log d = s$, as required.

  Let $(A', \phi')$ be a wasteful colouring differing in at most $t/\change$ trials from $(A, \phi)$, and let $F'_2 = \{(\phi(u), \phi(w)) \in F_2 : u,w\in A',~\phi'(u) = \phi(u), \text{ and } \phi'(w) = \phi(w)\}$.  Since $(A, \phi)\notin\sampleSpace^*_{v,c}$, for each $w \in N^2_G(v)$, we have $|\{u : (\phi(u), \phi(w)) \in F_2\}| < \log^2 d$, so $|F'_2| \geq |F_2| - t\log^2 d/\change$.   Let $F'\subseteq F$ be the induced subgraph of $F$ with bipartition $(N_F(F'_2), F'_2)$.  Since $\mu(G) \leq d^{1/4}$, each pair $(\phi(u), \phi(w))$ has degree at most $d^{1/4}$ in $F$, and thus $|N_F(F'_2)| \geq |F_1| - t$.  Hence,
  \begin{equation*}
    \uncolActivatedNbrsRV_{v,c}(A', \phi') \geq \uncolActivatedNbrsRV_{v,c}(A, \phi) - t,
  \end{equation*}
  and therefore $\uncolActivatedNbrsRV_{v,c}$ has upward $(s, \change)$-certificates, as claimed.

  Now by Theorem~\ref{BJ-tala} with $t = d^{5/6}/2$ and Claim~\ref{bad-event-prob}, we have
  \begin{multline}\label{uncoloured-degree-concentration}
    \Prob{\left|\uncolActivatedNbrsRV_{v,c} - \Expect{\uncolActivatedNbrsRV_{v,c}}\right| \geq \frac{1}{2}d^{5/6}} \\
    \leq 4\exp\left(-\frac{d^{10/6}}{256d^{3/2}\log^5 d}\right) + \Prob{\sampleSpace^*_{v,c}} \leq \frac{1}{2}\exp\left(-d^{1/7}\right).
  \end{multline}
  Now~\eqref{colour-degree-concentration} follows from~\eqref{active-degree-concentration} and~\eqref{uncoloured-degree-concentration}.  
\end{proof}

At this point we could use the Lov\' asz Local Lemma to prove a weaker form of Lemma~\ref{nibble-lemma} with $\Delta$ in the place of $\avgdeg$ and use this to obtain an arguably simpler proof of the result of Reed and Sudakov~\cite{ReSu02} generalised in two ways: to the setting of correspondence colouring and to the setting of multigraphs of bounded multiplicity.  The main simplification in the proof is the use of Theorem~\ref{BJ-tala}, the exceptional outcomes version of Talagrand's Inequality, to prove concentration of $\uncolActivatedNbrsRV$.

However, in order to prove Lemma~\ref{nibble-lemma} itself, we need to show that the $\phi$-available colours remaining for each vertex do not predominantly have much larger degree in $H$ than the average.  To that end, we introduce the following notation:
\begin{itemize}
\item we say a colour $c\in L(v)$ is \textit{relevant} if $\deg_H(c) \geq d/\log^3 d$, and
\item for each $v\in V(G)$, we let $L^{\mathrm{rel}}(v)$ be the set of relevant colours in $L(v)$.
\end{itemize}
For each $v\in V(G)$, we also define the random variables
\begin{itemize}
\item $\oldColourDegRV_v(A, \phi) = \sum_{c\in\availableList(v)}\deg_H(c)$ and
\item $\relColourDegRV_v(A, \phi) = \sum_{c\in L^{\mathrm{rel}}(v)\setminus \availableList(v)}\deg_H(c)$.
\end{itemize}

Our aim is to prove that $\oldColourDegRV_v$ is concentrated for each vertex $v$, but first we show that $\relColourDegRV_v$ is concentrated.
\begin{claim}\label{relevant-colour-degree-sum-lemma}
For $d$ sufficiently large, every vertex $v\in V(G)$ satisfies
  \begin{multline}
    \label{relevant-colour-degree-sum-concentration}
    \Prob{|\relColourDegRV_v - \Expect{\relColourDegRV_v}| > d^{11/6}} \\ \leq \exp\left(-d^{1/7}\right).
  \end{multline}
\end{claim}
\begin{proof}
  We apply Theorem~\ref{MR-tala} to $\relColourDegRV_v$ with $\change = d^{5/4}\log d$, $r = 2\log^3 d / d$, and $t = d^{11/6}$.  We bound $\Expect{\relColourDegRV_v} \le \Lambda d\le 4d^2$, so that $t/2 \geq 20\change + \sqrt{r\Expect{\relColourDegRV_v}} + 64\change^2 r$, as required.  
  
  If we change the outcome of a single trial, then $\relColourDegRV_v$ is most affected if we change the colour of a vertex $u$ to a colour $c$ such that $|N_H(c) \cap L^\mathrm{rel}(v)| = \mu(H)$ and moreover each colour $c' \in N_H(c) \cap L^\mathrm{rel}(v)$ has degree $\Delta(H) = d\log d$ in $H$.  Thus, changing the outcome of a trial affects $\relColourDegRV_v$ by at most $\change$, as required.  
  
  If $\relColourDegRV_v(A, \phi) \geq s$, then there is a set of at most $s\log^3 d / d$ colours in $L^{\mathrm{rel}}(v)\setminus L^\phi(v)$, and for each such colour $c$, there is a vertex $u\in A$ such that $c \in N_H(\phi(u))$.  Thus, the trials determining that $u \in A$ and $\phi(u)$ certify that $\relColourDegRV_v(A, \phi) \geq s$, and there are at most $2s\log^3 d / d = rs$ of them, as required.  Therefore by Theorem~\ref{MR-tala}, 
  \begin{multline*}
    \Prob{|\relColourDegRV_v - \Expect{\relColourDegRV_v}| > d^{11/6}} \\ \leq 4\exp\left(\frac{-d^{22/6}}{64d^{5/2}(\log^5 d)d^{-1}(4d^2 + d^{11/6})}\right) 
    \leq \exp\left(-d^{1/7}\right). \qedhere
  \end{multline*}
\end{proof}

Now we use Claim~\ref{relevant-colour-degree-sum-lemma} to prove that $\oldColourDegRV_v$ is concentrated for every vertex $v\in V(G)$.

\begin{claim}\label{colour-degree-sum-lemma}
  Every vertex $v\in V(G)$ satisfies
  \begin{equation}
    \label{expected-colour-degree-sum}
    \Expect{\oldColourDegRV_v} \leq d\cdot \Expect{\remainingColoursRV}
  \end{equation}
  and, for $d$ sufficiently large,
  \begin{multline}
    \label{old-colour-degree-sum-concentration}
    \Prob{\oldColourDegRV_v >  d\cdot \Expect{\remainingColoursRV} + \frac{2d\Lambda}{\log^3 d}}\\ \leq \exp\left(d^{-1/7}\right).
  \end{multline}
\end{claim}
\begin{proof}
  First we prove~\eqref{expected-colour-degree-sum}.
  By Linearity of Expectation,
  \begin{equation*}
    \Expect{\oldColourDegRV_v} = \sum_{c\in L(v)}\keep(v, c)\deg_H(c).
  \end{equation*}
  Recall $\keep(v, c) = (1 - p/\Lambda)^{\deg_H(c)}$.  By treating $\deg_H(c)$ for each $c\in L(v)$ as a real-valued variable bounded by $d\log d$ with sum at most $d\Lambda$, we can use Jensen's Inequality to show that the righthand side of the equality above is maximised when these variables are all equal to $d$.  See also~\cite[Lemma~5.2]{MoTh12} for a discrete version of this argument.  Thus, the righthand side of the above equality is at most $(1 - p/\Lambda)^dd\Lambda$, which by~\eqref{convexity-of-keep} is at most $d\cdot \sum_{c\in L(v)}\keep(v, c) = d\cdot \Expect{\remainingColoursRV_v}$, as desired.

  Now we prove~\eqref{old-colour-degree-sum-concentration}.
  By Claim~\ref{relevant-colour-degree-sum-lemma}, it suffices to show that if $(A, \phi)$ is a wasteful colouring such that
  \begin{equation}\label{if-relevant-colour-degrees-small-eq}
    \relColourDegRV_v(A, \phi) \geq \Expect{\relColourDegRV_v} - d^{11/6},
  \end{equation}
  then $\oldColourDegRV_v(A, \phi) \leq d\cdot \Expect{\remainingColoursRV_v} + 2d\Lambda/\log^3 d$.
  
  For this, first we have
  \begin{align}
    \label{relevant-colour-degree-sum-bound}
    &\Expect{\relColourDegRV_v} \nonumber\\
    & \geq \sum_{c\in L(v)}\deg_H(c) - \sum_{c\in L(v)\setminus L^{\mathrm{rel}}(v)}\deg_H(c) - \Expect{\oldColourDegRV_v} \nonumber\\
    & \ge \sum_{c\in L(v)}\deg_H(c) - \frac{d\Lambda}{\log^3 d} - d\cdot \Expect{\remainingColoursRV_v},
  \end{align}
  where in the last line we used~\eqref{expected-colour-degree-sum} and the definition of relevant.
  We also have
  \begin{multline}
    \label{lost-more-than-relevant-inequality}
    \oldColourDegRV_v(A, \phi) \leq \sum_{c\in L(v)}\deg_H(c) \\ - \relColourDegRV_v(A, \phi).
  \end{multline}
  Combining~\eqref{if-relevant-colour-degrees-small-eq}--\eqref{lost-more-than-relevant-inequality}, we have
  \begin{align*}
    \oldColourDegRV_v(A, \phi) \leq  d\cdot\Expect{\remainingColoursRV_v} + \frac{d\Lambda}{\log^3 d} + d^{11/6},
  \end{align*}
  and~\eqref{old-colour-degree-sum-concentration} follows.
\end{proof}

At last we have all the ingredients ---Claims~\ref{expectations-lemma}, \ref{concentration-lemma} and~\ref{colour-degree-sum-lemma}--- necessary to prove Lemma~\ref{nibble-lemma} via the Lov\' asz Local Lemma, as follows.  Recall that $\log^{-1} d \geq p \geq \log^{-2}d$.

\begin{proof}[Proof of Lemma~\ref{nibble-lemma}]
  First, rather than showing $|L'(v)| = \lceil(1 - p/(1 + 3\eps/4))\Lambda\rceil$ for every $v\in V(G)\setminus C$, it suffices to show that $|L'(v)| \geq (1 - p/(1 + 3\eps/4))\Lambda$, since we can truncate until equality holds by removing from $H'$ those colours $c \in L'(v)$ for which $|N_{H'}(c) \cap (V(H')\setminus L'(v))|$
  is largest, without increasing the maximum average colour degree.
  
  We sample a wasteful $(L, H)$-colouring by way of the wasteful random colouring procedure with activation probability $p$, as described earlier.
  We define the following set of bad events for each vertex $v\in V(G)$ and $c\in L(v)$:
  \begin{align*}
    \mathcal A_v &= 
                     \left\{
                       (A, \phi) \,:\, \frac{\remainingColoursRV_v(A, \phi)}{\Expect{\remainingColoursRV_v}} < 1 - p^{5/4}\right\},\\
    \mathcal A_{v, c} &= \left\{(A, \phi) \,:\, \frac{\colActivatedNbrsRV_{v,c}(A, \phi)}{\max\{p\cdot \deg_H(c), d^{6/7}\}} < \frac{1}{1  + \eps/8}\right\}, \text{ and}\\
    \mathcal A'_{v} &=
                        \left\{
                      (A, \phi) \,:\, \frac{\oldColourDegRV_{v}(A, \phi)}{d\cdot\Expect{\remainingColoursRV_v}} > 1 + p^{5/4}\right\}.
  \end{align*}
  Letting $\mathcal A$ be the union of all such bad events,
   note that each event in $\mathcal A$ is mutually independent of all but at most $(\Lambda d\log d)^4$ other events in $\mathcal A$, by our assumption that $\Delta(G) \leq \Lambda d\log d$.  
  By Claims~\ref{expectations-lemma} and~\ref{concentration-lemma}, using the fact that $p \geq \log^{-2} d$, for every $v\in V(G)$ we have
  \begin{equation}
    \label{useable-colours-concentration-equation}
    \Prob{\mathcal A_v} \leq \exp\left(-d^{1/7}\right),
  \end{equation}
  and for every $c\in L(v)$ we have
  \begin{equation}
    \label{coloured-neighbors-concentration-equation}
    \Prob{\mathcal A_{v,c}} \leq \exp\left(-d^{1/7}\right).
  \end{equation}
  Moreover by Claim~\ref{colour-degree-sum-lemma}, for every $v\in V(G)$ we have
  \begin{equation}
    \label{colour-degree-sum-concentration-equation}
    \Prob{\mathcal A'_v} \leq \exp\left(-d^{1/7}\right).
  \end{equation}
  Therefore by~\eqref{useable-colours-concentration-equation}--\eqref{colour-degree-sum-concentration-equation} and the Lov\' asz Local Lemma, there is a wasteful colouring $(A, \phi)\notin \mathcal A$, for all sufficiently large $d$.

  Now we show that $\phi|_{\colouredVtcs}$ satisfies the conclusion of Lemma~\ref{nibble-lemma} where $H' = H^\phi - \cup_{v\in \colouredVtcs}L(v)$ and $L' = L^\phi|_{V(G)\setminus\colouredVtcs}$.  Since $p \leq \log^{-1} d$ and we assume $d$ is sufficiently large, we may assume $p$ is sufficiently small for certain inequalities to hold.
  Indeed, for small enough $p$, every vertex $v\in V(H')$ satisfies
  \begin{equation*}
    |L'(v)| \geq \left(1 - p^{5/4}\right)\left(1 - \frac{p}{1 + \eps}\right)\Lambda \geq \left(1 - \frac{p}{1 + 3\eps/4}\right)\Lambda,
  \end{equation*}
  as required, where we used Claim~\ref{expectations-lemma} and the fact that $(A, \phi)\notin \mathcal A_v$.  Moreover, each vertex $v\in V(H')$ satisfies
  \begin{multline*}
    \frac{1}{|L'(v)|}\sum_{c\in L'(v)}\deg_{H'}(c) \leq \frac{1}{|L'(v)|}\sum_{c\in L'(v)}\deg_H(c) - \frac{\max\{p\cdot \deg_H(c), d^{6/7}\}}{1 + \eps / 8}\\
    \leq \frac{1}{|L'(v)|}\left(1 - \frac{p}{1 + \eps/8}\right)\sum_{c\in L'(v)}\deg_H(c) - d^{6/7},
  \end{multline*}
  where we used the fact that $(A, \phi)\notin \mathcal A_{v,c}$ for any $c\in L'(v)$.
  And since 
  $(A, \phi)\notin \mathcal A_v$ and $(A, \phi)\notin \mathcal A'_v$ for any $v\in V(G)$,
  we have
\begin{equation*}
  \frac{1}{|L'(v)|}\sum_{c\in L'(v)}\deg_{H'}(c) \leq \frac{\left(1 - \frac{p}{1 + \eps/8}\right)\left(1 + p^{5/4}\right)}{1 - p^{5/4}}\cdot d - d^{6/7}.
\end{equation*}
Now $(1 + p^{5/4})/(1 - p^{5/4}) < 1 + 3p^{5/4}$ for small enough $p$, and the righthand side above
is at most $\left(1 - \frac{p}{1 + \eps/4}\right)d$.
  Thus $\Delta_{L'}(H') \leq \left(1 - \frac{p}{1 + \eps/4}\right)d$, as desired.
\end{proof}

\section{Conclusion}\label{sec:conclusion}

We conclude with some perspectives for future research.

First it is conceivable that Theorem~\ref{thm:main} could be strengthened further.
For $H$ being the cover graph for $G$ via $L$, let us define {\em maximum average colour multiplicity $\avgmult_L(H)$ of $H$ with respect to $L$} by
\[
\avgmult_L(H) := \max_{vv' \in E(G)} \frac{1}{|L(v)|}\sum_{w\in L(v)} |N_H(w) \cap L(v')|.
\]
Note that $\avgmult_L(H) \le \mu_L(H)$ always.
We believe that the statement of Theorem~\ref{thm:main} also holds with $\avgmult_L(H)$ in the place of $\mu_L(H)$.
This would imply a conjecture of Loh and Sudakov~\cite[p.~917]{LoSu07} in a stronger form: they posited this with $\Delta(H)$ instead of $\avgdeg_L(H)$.

Second $\Lambda'$ in Question~\ref{avgquestion} in general lies between $2d$ and $4d$, as noted in the introduction, but its sharper determination remains a tempting problem.

Last we contend that many questions on independent transversals and colourings in terms of $\avgdeg_L(H)$ instead of $\Delta(H)$ may give rise to interesting challenges. Indeed, this work was partially motivated by such a study in terms of graphs embeddable in surfaces of prescribed genus~\cite[Theorem~1]{DEKO18+}.

\subsection*{Note added}

During the preparation of this manuscript, we learned of the concurrent and independent work of Glock and Sudakov~\cite{GlSu20+}. They also proved Theorem~\ref{thm:main} with a similar method. In their proof, they provided a weaker form of our Theorem~\ref{semirandom-partial-colouring-thm} and established a more efficient form of our Theorem~\ref{list-reduction-lemma}. This demonstrates considerable slack in the method and suggests that further refinement could lead to new developments.

Glock and Sudakov's work also had differing underlying motivation, more from independent transversals than from graph colouring. They proved Theorem~\ref{thm:main} as a means toward the solution of certain problems about independent transversals (of which we had been unaware), especially one due to Erd\H{o}s, Gy\'arf\'as and {\L}uczak~\cite{EGL94}, from a quarter of a century ago.

Because it is brief, we include this easy application for the benefit of the reader. Erd\H{o}s, Gy\'arf\'as and {\L}uczak~\cite{EGL94} asked for the determination of $f(k)$, the least $n$ such that, for any graph $H$ on $nk$ vertices having a partition $L$ into parts of size $k$ such that each bipartite subgraph induced between two distinct parts has no more than one edge, there is guaranteed to be an independent transversal. They showed that $k^2/(2e) \le f(k) \le (1+o(1))k^2$ as $k\to\infty$. Note that every $H$ and $L$ as above satisfies $\avgdeg_L(H) \le (n-1)/k$. For a lower bound on $f(k)$, it suffices by Theorem~\ref{thm:main} to choose a suitable $n=n(k)$ satisfying that $k \ge (1+o(1))(n-1)/k$ as $k\to\infty$. This yields $f(k) \ge (1+o(1))k^2$ as $k\to\infty$, matching their original upper bound, and settling their problem asymptotically.

\bibliographystyle{abbrv}
\bibliography{transversals}

\appendix
\renewcommand\appendixname{A}
\renewcommand{\thesection}{\Alph{section}}
\section{Proof of Theorem~\ref{list-reduction-lemma}}

In this section we prove Theorem~\ref{list-reduction-lemma}.  First, we need the following ``random halving'' lemma, whose proof is similar to that of~\cite[Lemma 3.3]{LoSu07}.
\begin{lemma}\label{list-halving-lemma}
  If $H$ is a cover graph for $G$ via $L$ satisfying 
  \begin{itemize}
  \item $|L(v)| = 2s$ for all $v\in V(G)$,
  \item $\avgdeg_L(H) \leq d$,
  \item $\Delta_L(H) \leq d\log d$, and
  \item $\mu_L(H) \le \mu$ for some $\mu > \log^{10} d$,
  \end{itemize}
  where $s \geq d / 2$ and $d$ is sufficiently large, then there exists an induced subgraph $H'\subseteq H$ that is a cover graph for $G$ via $L'$ for some $L'$ satisfying
  \begin{itemize}
  \item $|L'(v)| = s$ for all $v\in V(G)$,
  \item $\avgdeg_{L'}(H') \leq d/2 + d^{3/5}$,
  \item $\Delta_{L'}(H') \leq (d\log d)/2 + d^{3/5},$ and
  \item $\mu_{L'}(H') \leq \mu/2 + \mu^{3/5}$.
  \end{itemize}
\end{lemma}
\begin{proof}
  For each $v \in V(G)$, arbitrarily enumerate the colours in $L(v)$ as $c_1, \dots, c_{2s}$ such that $\deg_H(c_1) \leq \cdots \leq \deg_H(c_{2s})$.  For each $i \in \{1, \dots, s\}$, we say that $c_{2i - 1}$ and $c_{2i}$ are \textit{mates at $v$}, and we randomly and independently designate one of $c_{2i - 1}$ and $c_{2i}$ to be in $L'(v)$.  Let $H'$ be the cover graph for $G$ via $L'$.  Clearly $|L'(v)| = s$ for each $v \in V(G)$, as required.  For each vertex $v \in V(G)$ and $c \in L(v)$, let $\mathcal A_{v, c}$ be the event that $\deg_{H'}(c) > \deg_H(c)/2 + d^{4/7}$, and for each $u \in N_G(v)$, let $\mathcal B_{v, c, u}$ be the event that $|N_{H'}(c) \cap L'(u)| > \mu/2 + \mu^{4/7}$.  We use the Lov\'asz Local Lemma to prove that with nonzero probability none of these events occur.

  To that end, we bound the probabilities of $\mathcal A_{v, c}$ and $\mathcal B_{v, c, u}$.  For each $v \in V(G)$ and $c \in L(v)$, partition $N_H(c)$ into $S_{v,c}$ and $T_{v,c}$, where $(u, c') \in N_H(c)$ is in $T_{v,c}$ if the mate of $c'$ at $u$ is not in $N_H(c)$, and otherwise $(u, c')$ is in $S_{v, c}$. Clearly $\deg_{H'}(c) = |S_{v,c}\cap V(H')|+|T_{v,c}\cap V(H')|$.
    Now precisely half of $S_{v, c}$ is in $H'$, and the number of members of $T_{v, c}$ in $H'$ is the sum of $|T_{v, c}|$ independently distributed indicator random variables with mean $1/2$.  Therefore the Chernoff Bound implies that
  \begin{equation*}
    \Prob{\mathcal A_{v, c}} \leq 2\exp\left(-\frac{d^{8/7}}{d\log d}\right) \leq \exp(-d^{1/8})
  \end{equation*}
  provided $d$ is sufficiently large.
  A similar argument, using that $\mu > \log^{10} d$, implies that for every $u \in N_G(v)$,
  \begin{equation*}
    \Prob{\mathcal B_{v, c, u}} \leq 2\exp\left(-\frac{\mu^{8/7}}{\mu}\right) \leq \exp(-\log^{10/7}d) \leq d^{-10}
  \end{equation*}
  provided $d$ is sufficiently large.
  Each of these events is mutually independent of all but at most $d^3$ other events, so by the Lov\'asz Local Lemma, we may assume from here on that none of the events $\mathcal A_{v, c}$ or $\mathcal B_{v, c, u}$ happen.

  Clearly $H'$ and $L'$ satisfy $\Delta_{L'}(H) \leq \Delta_L(H)/2 + d^{4/7} \leq (d\log d) / 2 + d^{3/5}$ and $\mu_{L'}(H') \leq \mu/2 + \mu^{3/5}$, as required.  It only remains to show that $\avgdeg_{L'}(H') \leq d / 2 + d^{3/5}$. To see this, first note that for each vertex $v \in V(G)$, by the choice of the mates at $v$, we have
  \begin{multline*}
    \sum_{i=1}^s \deg_H(c_{2i}) = \sum_{i=1}^{2s}\deg_H(c_i) - \sum_{i=1}^s\deg_H(c_{2i - 1})\\
    \leq 2sd + \deg_H(c_{2s})- \sum_{i=1}^{s}\deg_H(c_{2i}),
  \end{multline*}
  which implies that
  \begin{equation*}
    \sum_{i=1}^s \deg_H(c_{2i}) \leq sd + d\log d.
  \end{equation*}
  Moreover, since $\mathcal A_{v, c}$ does not hold for any $c \in L(v)$, we have from the definition of $L'$ that
  \begin{multline*}
    \sum_{c \in L'(v)}\deg_{H'}(c) \leq \sum_{c \in L'(v)} (\deg_H(c) / 2 + d^{4/7}) \\
    \leq (1/2)\sum_{i=1}^s \deg_H(c_{2i}) + d\log d/2 + sd^{4/7}
    \leq sd / 2 + sd^{4/7}.
  \end{multline*}
  In particular, since $d \leq 2s$, we have $\avgdeg_{L'}(H) \leq (sd / 2 + d\log d/2 + sd^{4/7}) / s \leq d / 2 + d^{3/5}$, as desired.
\end{proof}

We also need the following proposition (similar to~\cite[Proposition~4.1]{MoTh12}) to reduce the maximum colour degree by removing a negligible number of colours for each vertex.

\begin{proposition}\label{trim-large-degree-from-list-prop}
  If $H$ is a cover graph for $G$ via $L$ satisfying
  \begin{itemize}
  \item $|L(v)| \geq (1 + \eps)d$ and
  \item $\avgdeg_L(H) \leq d$, 
  \end{itemize}
  where $d$ is sufficiently large, then
  there exists an induced subgraph $H'\subseteq H$ that is a cover graph for $G$ via $L'$ for some $L'$ satisfying
  \begin{itemize}
  \item $|L(v)| \geq (1 + 9\eps/10)d$,
  \item $\avgdeg_{L'}(H') \leq d$, and
  \item $\Delta_{L'}(H') \leq d\log^{1/2} d$.
  \end{itemize}
\end{proposition}
\begin{proof}
  For every vertex $v\in V(G)$, at most $|L(v)|d / (d\log^{1/2}d) \leq |L(v)| / \log^{1/2}d$ colours $c\in L(v)$ satisfy $\deg_H(c) > d\log^{1/2}d$. Let $H'\subseteq H$ and $L'$ be induced by removing all such colours for every $v\in V(G)$, and such that $H'$ is a cover graph for $G$ via $L'$.
 Now every $v\in V(G)$ satisfies $|L'(v)| \geq |L(v)|(1 - \log^{-1/2}d) \geq (1 + \eps)d(1 - \log^{-1/2}d) \geq (1 + 9\eps / 10)d$ for $d$ sufficiently large, as desired.  Moreover, $\avgdeg_{L'}(H') \leq \avgdeg_L(H) \leq d$ and $\Delta_{L'}(H') \leq d\log^{1/2}d$, as desired.
\end{proof}

Now we prove Theorem~\ref{list-reduction-lemma}.  The proof is similar to that of~\cite[Theorem 3.1]{LoSu07}.

\begin{proof}[Proof of Theorem~\ref{list-reduction-lemma}]
  Let $H$ be a cover graph for $G$ via $L$ satisfying $|L(v)| \geq (1 + \eps)d$ for all $v\in V(G)$, $\avgdeg_L(H) \leq d$, and $\mu_L(H) \leq \gamma d$.  We may assume $\eps < 1$, and we assume $\gamma^{-1}$ and $d$ are sufficiently large for various inequalities to hold throughout the proof.  By Proposition~\ref{trim-large-degree-from-list-prop}, there is an induced subgraph $H_0 \subseteq H$ that is a cover graph for $G$ via $L_0$ for some $L_0$ satisfying $|L(v)| \geq (1 + 9\eps/10)d$, $\avgdeg_{L'}(H') \leq d$, and $\Delta_{L'}(H') \leq d\log^{1/2} d$.  We may assume $\gamma^{-6/5} < d$, or else $H_0$ and $L_0$ satisfy all the required properties with $d' = d$, and the result follows.

  Let $j \geq 1$ be the integer for which $2^{j - 1} < \gamma^{6/5}d \leq 2^j$.  By deleting at most $2^j$ colours from $L_0(v)$ for each $v\in V(G)$ (choosing only those colours $c \in L_0(v)$ for which $|N_{H_0}(c) \cap (V(H_0)\setminus L_0(v))|$ is largest to remove so as not to increase the maximum average colour degree), we may assume without loss of generality that $|L_0(v)| = s_0$ for $s_0 \geq (1 + 9\eps / 10)d - 2^j$ divisible by $2^j$.  Now let $d_0 = d$, let $\mu_0 = \mu$, and define the sequences $\{d_t\}$, $\{s_t\}$, and $\{\mu_t\}$ where
  \begin{align*}
    d_{t + 1} = d_t / 2 + d_t^{2/3}, && s_{t + 1} = s_t / 2, && \mathrm{and} && \mu_{t + 1} = \mu_t / 2 + \mu_t^{2/3}.
  \end{align*}
  We claim that
  \begin{enumerate}[(i)]
  \item\label{reduction:degree-bound} $\gamma^{-6/5} / 2 < d_j \leq (1 + \eps / 10)d / 2^j$
  \item\label{reduction:multiplicity-ub} $\mu_j \leq 8d_j^{1/6}$, and
  \item\label{reduction:multiplicity-lb} $\mu_t > \log^{10} d_t$ for all $t \in \{0, \dots, j - 1\}$.
  \end{enumerate}
  The proof is essentially the same as in~\cite[Theorem 3.1]{LoSu07}, so we omit it.  (For~\ref{reduction:multiplicity-ub}, the argument in~\cite[Theorem 3.1]{LoSu07} yields $\mu_j \leq 8\gamma d / 2^j$, and since $\gamma \leq (2^j / d)^{5/6}$, we have $\mu_j \leq 8 d^{1/6} / 2^{j / 6} \leq 8 d_j^{1/6}$, as desired.)

  Now by~\ref{reduction:multiplicity-lb}, we can apply Lemma~\ref{list-halving-lemma} $j$ times to obtain a sequence $H_0 \supseteq H_1 \supseteq \cdots \supseteq H_j$ of induced subgraphs of $H$, where $H_t$ is a cover graph for $G$ via $L_t$ for some $L_t$ satisfying $|L_{t}(v)| = s_t$ for all $v \in V(G)$, $\avgdeg_{L_t}(H_t) \leq d_t$, $\Delta_{L_t}(H_t) \leq (d_{t - 1} \log d_{t - 1}) / 2 + d_{t - 1}^{3/5} \leq d_t \log d_t$, and $\mu_{L_t}(H_t) \leq \mu_t$ for every $t \in \{1, \dots, j\}$.
  In particular, $\avgdeg_{L_j}(H_j) \leq d_j$ and by~\ref{reduction:multiplicity-ub}, $\mu_{L_j}(H_j) \leq d_j^{1/5}$, as required.
  By the lower bound of~\ref{reduction:degree-bound}, we can make $d_j$ as large as necessary by assuming $\gamma^{-1}$ is large, so the upper bound of~\ref{reduction:degree-bound} yields
  \begin{equation*}
    s_j \geq \frac{(1 + 9\eps / 10)d - 2^j}{2^j} \geq \frac{1 + 9\eps / 10}{1 + \eps / 10}d_j - 1 \geq (1 + 7\eps / 10)d_j.
  \end{equation*}
  Finally, by applying Proposition~\ref{trim-large-degree-from-list-prop} to $H_j$ and $L_j$, we obtain a cover graph $H'$ for $G$ via $L'$ for some $L'$ satisfying $|L'(v)| \geq (1 + 9(7\eps/10)/10)d_j \geq (1 + \eps / 2)d_j$ for every $v\in V(G)$, $\avgdeg_{L'}(H') \leq d_j$, and $\Delta_{L'}(H') \leq d_j\log^{1/2}d_j$, as desired.
\end{proof}

\end{document}